\documentclass[12pt, a4wide]{article}

\usepackage{amsmath,amssymb,amsfonts,amsthm}
\usepackage{amsthm}
\usepackage{graphicx}
\usepackage[a4paper]{geometry}
\usepackage{fouriernc}
\usepackage[T1]{fontenc}
\usepackage[]{algorithm2e}
\usepackage{graphicx}    

\newtheorem{theorem}{Theorem} 	
\newtheorem{lemma}{Lemma}
\newtheorem{proposition}{Proposition}
\newtheorem{conjecture}{Conjecture}

\theoremstyle{definition}
\newtheorem{definition}{Definition}
\newtheorem{remark}{Remark}

\DeclareMathOperator{\tr}{\operatorname{tr}}
\DeclareMathOperator{\E}{\mathbb{E}}

\newcommand\scalemath[2]{\scalebox{#1}{\mbox{\ensuremath{\displaystyle #2}}}}
\graphicspath{{./Pics}}

\title{Evidence of the Poisson/Gaudin--Mehta phase transition for banded matrices on global scales}
\author{Sheehan Olver\thanks{School of Mathematics and Statistics, The University of Sydney, NSW 2006, Australia (\texttt{sheehan.olver@sydney.edu.au})} \and 
Andrew Swan\thanks{Department of Applied Mathematics and Theoretical Physics, University of Cambridge, UK, (\texttt{acks2@cam.ac.uk})}}
\begin{document}
	\maketitle
		\begin{abstract}
	We prove that the Poisson/Gaudin--Mehta phase transition conjectured to occur when the bandwidth of an $N \times N$ symmetric banded matrix grows like $\sqrt N$  is observable as a critical point in the fourth moment of the level density for a wide class of symmetric banded matrices.  A second critical point when the bandwidth grows like ${2 \over 5} N$ leads to a new conjectured phase transition in the eigenvalue localization, whose existence we demonstrate in numerical experiments.  
	\end{abstract}
	\section{Introduction}
		An important open problem in random matrix theory is the {\it Poisson/Gaudin--Mehta  conjecture} on the existence of  a phase transition in the local eigenvalue statistics of random real symmetric banded (RSB) matrices:
	\begin{conjecture}
		The limiting local statistics of a random RSB matrix with independent entries and bandwidth $b \asymp N^\alpha$ are Poisson if $\alpha < \frac{1}{2}$ and are Gaudin--Mehta if $\alpha > \frac{1}{2}$.
	\end{conjecture}
	This paper concerns the question of whether this phase transition is also observable in the global eigenvalue statistics, e.g., the level density of the eigenvalues. 
	\begin{definition}
		An $N\times N$ random matrix $H$ is a \emph{real symmetric banded Wigner matrix} (or RSB Wigner matrix) with \emph{bandwidth} $b$ if the following four conditions hold:
		\begin{enumerate}
			\item $H$ is a real symmetric matrix.
			\item The upper triangular elements $h_{ij}$, $i \le j$, are jointly independent real random variables.
			\item For $|i -j|<b$, $\E h_{ij} = 0$ and $\E h_{ij}^2 = 1$.
			\item For $|i-j| \ge b$, $h_{ij} = 0$.
		\end{enumerate}
	\end{definition}
	The limiting level densities of RSB Wigner matrices were first considered in 1991 by Bogachev, Khorunzhii, Molchanov, and Pastur \cite{Bogachev1991, Molchanov1992}. In these papers, they obtained the following result:
	\begin{proposition}
		The normalized level density of an RSB Wigner matrix with bandwidth $b = cN^\alpha$ is given by the Wigner semi-circle in the limit $N \rightarrow \infty$ for any $0 < \alpha < 1$ and $0 < c \le 1$. If $\alpha = 1$ and $0 < c < 1$ or $b = \mathrm{const.}$, then some other distribution is obtained. 
	\end{proposition}
	This is surprising for two reasons: firstly, the Wigner semi-circle law extends down to small bandwidths $b = \mathcal{O}(N^\varepsilon)$, but even more surprisingly, it does not hold for very large bandwidths $b = cN$. It would therefore seem that the Poisson/Gaudin--Mehta transition is unobservable at a macroscopic scale, and indeed, the authors of \cite{Molchanov1992} make an interesting remark to this effect:
	\begin{quote}
		[The parameter $\frac{b^2}{N}$] does not play a special role in the formation of the density of states of the considered matrices. However, we are not inclined to regard our results as incompatible with the interesting [Poisson/Gaudin--Mehta conjecture], since, as is well known, the integrated density of states has very little sensitivity to the localization properties of the states of random operators.
	\end{quote}
	We contend that this is not the case: the key is that ``very little sensitivity'' is not the same as ``none''. Our main result is a proof that the Poisson/Gaudin--Mehta transition is indeed observable on the global scale, as a critical point in the fourth moment:
%
%
%
	\begin{theorem}
		In the limit $N\rightarrow \infty$, the fourth moment of the normalised level density 
		\begin{equation}\label{key}
		m_4(\sigma_{N,b}) = \int_{-\infty}^{\infty} x^4 \sigma_{N,b}(dx)
		\end{equation}
		of an RSB Wigner ensemble satisfying a four moment condition, has two critical points
		\begin{equation*}\label{key}
		\partial_b m_4(\sigma_{N,b}) = 0 
		\end{equation*}
		found at 
		\begin{equation*}\label{key}
		b = \left(\frac{3N}{2}\right)^\frac{1}{2} + o(N^{\frac{1}{2}})
		\end{equation*} 
		and
		\begin{equation*}\label{key}
		b = \frac{2}{5}N + o(N).
		\end{equation*}
	\end{theorem}
	\noindent We prove this theorem in Section~\ref{Section4.3}.  We further demonstrate in Section~\ref{newphase} numerical evidence that the second critical point at $b = \frac{2}{5}N + o(N)$  corresponds precisely to the location of a second (lower order)  localisation/delocalisation transition in the eigenvectors.   The existence of this transition has not been found in in the existing literature, either conjectured or proven.

	\section{Previous Results}
	\subsection{Local statistics}
	The Poisson/Gaudin--Mehta conjecture appears as the final entry in Percy Deift's list of open problems in integrable systems and random matrices \cite{Deift2007}.   Interest in this conjecture can be traced back to the numerical work of Seligman, Verbaarschot, and Zirnbauer \cite{Seligman1984}, and Casati, Izrailev, and Molinari \cite{Casati1991} in the late 1980s/early 1990s. Since this time, progress has been rather modest: results concerning local universality are only just beginning to emerge, and only apply to banded random matrices with large bandwidth. Shcherbina proved local universality for a specially structured `periodic block banded GUE' with large bandwidth $b = cN$ using a rigorous supersymmetric method \cite{Shcherbina2014a};  Bourgade, Erd\H{o}s, Yau, and Yin \cite{Bourgade2016} strengthened this result using heat flow methods combined with a new mean field reduction technique to prove local universality for more general periodic banded Wigner ensembles in both the real symmetric/Hermitian cases, again for $b = cN$. 
	
	Outside of numerical simulation and heuristic supersymmetric arguments (see Fyodorov and Mirlin \cite{Fyodorov1991}), the only results indicating the critical bandwidth $b =\sqrt{N}$ have concerned the two point correlation function of the characteristic polynomial
	\begin{equation}\label{key}
	F_2(\lambda_1, \lambda_2) = \E \left(\det(A-\lambda_1)\det(A-\lambda_2)\right),
	\end{equation}
	which, in an appropriate scaling limit goes to
	\begin{equation}\label{key}
	F_2(\lambda_1, \lambda_2) \rightarrow 1
	\end{equation}
	if $\alpha < \frac{1}{2}$ \cite{Shcherbina2014}, and goes to
	\begin{equation}\label{key}
	F_2(\lambda_1, \lambda_2) \rightarrow \frac{\sin(\lambda_1 -\lambda_2)}{\lambda_1 - \lambda_2}
	\end{equation}
	if $\alpha > \frac{1}{2}$ \cite{Shcherbina2016}. These results, although consistent with Poisson/Gaudin--Mehta local statistics, do not imply it. 
	
	On the other hand, much more is known at the edge of the spectrum: Sodin showed that the limiting distribution of the extreme eigenvalues are given by the Tracy--Widom distribution if $b \gg N^{\frac{5}{6}}$ and by some other distribution if $b \ll N^{\frac{5}{6}}$ \cite{Sodin2010}.
	\subsection{Eigenvector statistics of the BGE}
	The situation has been somewhat better on the eigenvector side of the story. Phrased in terms of the eigenvectors, the analogous conjecture is as follows:
	\begin{conjecture}
		The eigenvector localisation length $l$ for a random banded matrix with independent entries and bandwidth $b$ is $l \asymp b^2$. 
	\end{conjecture}
	Due to the connection with disordered conductors, this is also known as the Anderson transition conjecture for banded matrices. The first result concerning this conjecture was due to Schenker \cite{Schenker2009}, who provided an upper bound $l \le b^8$. A lower bound of $l\ge b^\frac{7}{6}$ was demonstrated by Erd\H{o}s and Knowles \cite{Erdos2010}; together with Yau and Yin \cite{Erdos2013}, they improved this to  $l\ge b^\frac{5}{4}$. These lower bounds hold in a weaker sense, in that they only hold for `most' eigenvectors. A strong lower bound of $l\ge b^\frac{7}{6}$ was shown by Bao and Erd\H{o}s for block banded Wigner ensembles \cite{Bao2015}.

	\section{Matrix Model}
			The \emph{one-point correlation function} or \emph{level density} of a Hermitian random matrix $H \sim \mathrm{E}_N$ is the unique measure $\rho_N: \mathcal{B}(\mathbb{R}) \rightarrow \mathbb{R}^+$ such that
			\begin{equation}\label{key}
			\int_{-\infty}^{\infty} f(x) \rho_N(dx) = \sum_{i = 1}^N \E f(\lambda_i)
			\end{equation}
			for all continuous compactly supported functions $f: \mathbb{R} \rightarrow \mathbb{R}$, where $\lambda_i$ are the $N$ ordered eigenvalues of $H$.
			For discrete ensembles, $\rho_N$ can only be defined as a measure on $\mathbb{R}$, but for continuous ensembles we can write $\rho_N(dx) = \rho_N(x)dx$ as a density. In this case, we equivalently have
			\begin{equation}\label{key}
			\rho_N(x) = N\int_{-\infty}^{\infty} \dots \int_{-\infty}^{\infty} P(x_1, \dots, x_N) dx_2\dots dx_N,
			\end{equation}
			where $P(x_1, \dots, x_N)$ is the joint density of the eigenvalues. From the one point correlation function, we construct the \emph{normalised one-point correlation function} or \emph{normalised level density} $\sigma_{N}: \mathcal{B}(\mathbb{R}) \rightarrow \mathbb{R}^+$ through the rescaling
			\begin{equation}\label{key}
			\sigma_N(x) = \nu \rho_N(\eta x),
			\end{equation}
			where the scaling parameters
			\begin{equation}\label{key}
			\eta = \sqrt{\frac{m_2(\rho_N)}{m_0(\rho_N)}} 
			\end{equation}
			and
			\begin{equation}\label{key}
			\nu = \sqrt{\frac{m_2(\rho_N)}{m_0(\rho_N)^3}} 
			\end{equation}
			are given in terms of the moments of the one point correlation function
			\begin{equation}\label{key}
			m_k(\rho_N) = \int_{-\infty}^{\infty}x^k \rho_N(dx) = \E\tr(H^k).
			\end{equation}
			This scaling is chosen so that
			\begin{equation}\label{key}
			m_0(\sigma_N) = m_2(\sigma_N) = 1,
			\end{equation}
			i.e., so that $\sigma_N$ is a probability measure with variance 1. From this scaling we therefore have
	\begin{equation}\label{key}
	m_k(\sigma_N) = \frac{m_0(\rho_N)^{\frac{k}{2}-1}m_k(\rho_N)}{m_2(\rho_N)^{\frac{k}{2}}} = \frac{N^{\frac{k}{2}-1}\E \tr (H^k)}{\left(\E \tr(H^2)\right)^\frac{k}{2}},
	\end{equation}
	and in particular, the fourth moment of $\sigma_N$ is
		\begin{equation}\label{key}
		m_4(\sigma_N) = \frac{N\E \tr (H^4)}{\left(\E \tr(H^2)\right)^2}.
		\end{equation}

		In this paper, we consider RSB Wigner matrices that satisfy a four moment condition:
		\begin{definition}
			An $N\times N$ random matrix $H$ is a \emph{four moment Gaussian matching RSB Wigner matrix} with \emph{bandwidth $b$} if $H$ is an RSB Wigner matrix and for $|i -j|<b$, the first four moments of $h_{ij}$ match those of a standard Gaussian random variable, i.e additionally $ \E h_{ij}^3 = 0$, $\E h_{ij}^4 = 3$.
			The ensemble associated to such a random matrix we will denote by $\mathrm{BWE}_{N,b}$, i.e., so that $H \sim \mathrm{BWE}_{N,b}$.
		\end{definition}
	\section{Critical behaviour on the global scale}\label{Section4.3}
	In terms of moments, Proposition 1 now states that in the limit $N \rightarrow \infty$, $b = N^\alpha$, the moments of normalised level density $m_k(\sigma_{N,b})$ of an RSB Wigner ensemble will converge to those of the Wigner semi-circle $m_{k}(\mu_{\mathrm{sc}})$. What about the moments for finite $N$ and $b$? Figure \ref{fig:m4samp} shows a numerical simulation of $m_4(\sigma_{N,b})$.
	\begin{figure}[t]
		\centering
		\includegraphics[scale=1]{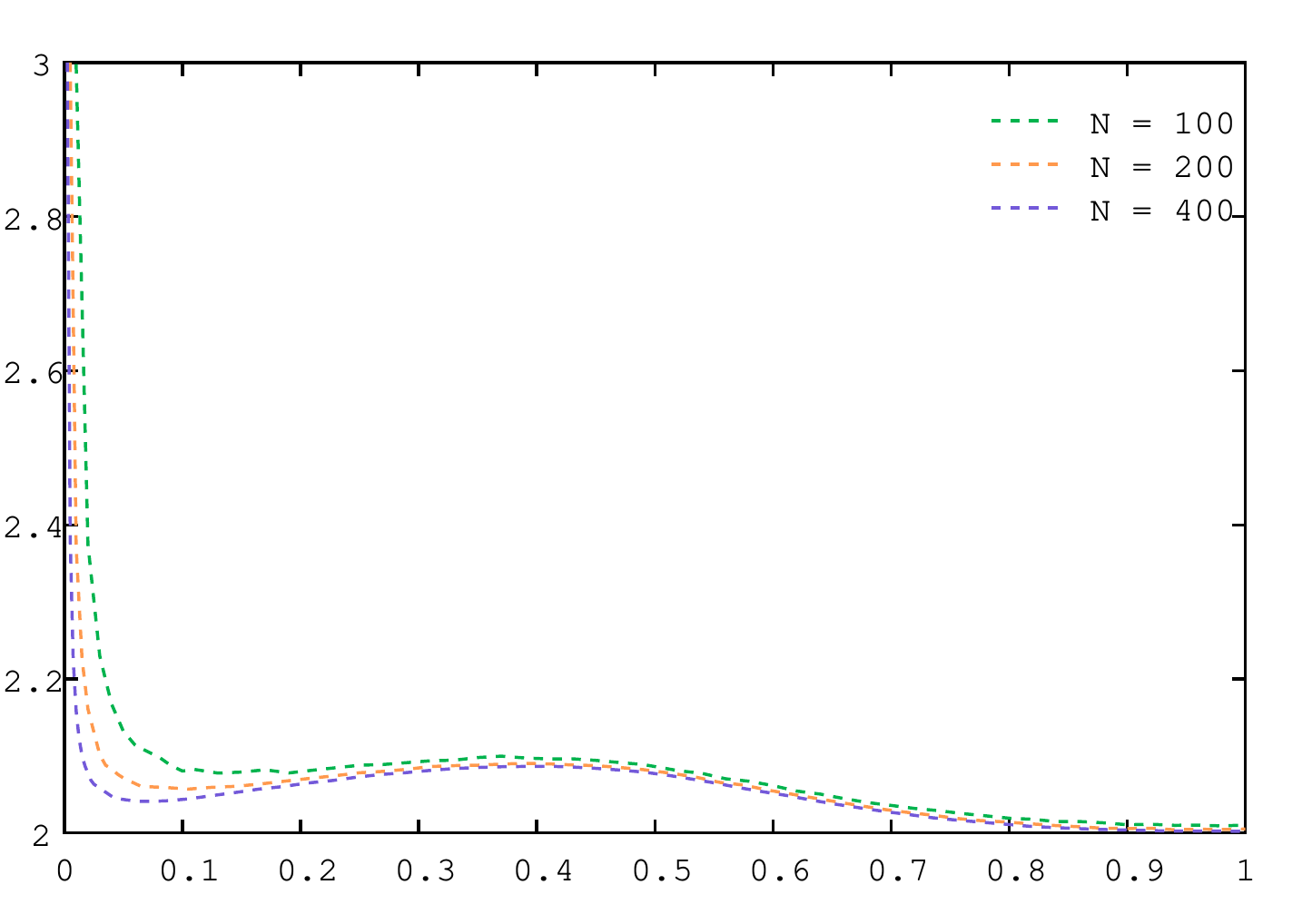}
		\caption{This figure shows $m_4(\sigma_{N,b})$ against $c$ with $b = cN$ for $N = 100,200,400$. Two critical points are evident: the first near $b \simeq \sqrt{N}$, and the second near $b \simeq \frac{2}{5}N$ }
		\label{fig:m4samp}
	\end{figure}
	Two critical points can be seen, which as we prove below, persist in the limit $N \rightarrow \infty$ after rescaling.
	
	\begin{remark}
		In the following, it will be convenient to take $N = mb$ as an integer multiple of the bandwidth $b$, where $m \in \mathbb{N}$ is not necessarily fixed. In this case, the matrix $H \sim \mathrm{BWE}_{N,b}$ can be represented in block diagonal form as
		\begin{equation}\label{blockform}
		H = \begin{bmatrix}
		A_1 & L_1 \\
		L_1^\top & A_2 & L_2\\
		& L_2^\top & A_3 & \ddots\\
		& & \ddots & \ddots& L_{m-1}\\
		& & & L_{m-1}^\top & A_m\\
		\end{bmatrix},
		\end{equation}
		where the $A_i$ are $b\times b$ random matrices drawn from $\mathrm{BWE}_{b,b}$, and the $L_i$ are strictly lower triangular $b \times b$ random matrices with independent entries $l_{ij}$ below the diagonal. We may consider these lower triangular blocks $L_i$ as drawn from an ensemble $\mathrm{LTE}_{b}$, defined in the obvious way.
	\end{remark}
	\begin{remark}
		To simplify notation, we drop subscripts to indicate a generic element.
	\end{remark}
	\begin{lemma}
		Let $A \sim \mathrm{BWE}_{b,b} $ and $L_1,L_2 \sim \mathrm{LTE}_{b}$ be $b \times b$ jointly independent random matrices. Then
		\begin{flalign}
		\E \tr(A^4) &= 2b^3 + b^2 \label{A4}\\
		\E \tr (A^2LL^\top) &= \frac{b^3 - b^2}{2} \label{A2LLT}\\
		\E \tr (A^2L^\top L) &= \frac{b^3 - b^2}{2} \label{A2LTL}\\
		\E \tr(L_1^\top L_1 L_2 L_2^\top)  &= \frac{b^3 - 3b^2 + 2b}{6} \label{L1L2}\\
		\E \tr(L L^\top L L^\top) &=  \frac{4b^3 -3b^2 -b}{6}\label{LLLL}.
		\end{flalign}  
	\end{lemma}
	\begin{proof}
		The proof requires a straight forward, but tedious, moment calculation. Details may be found in the appendix. 
	\end{proof}
	
	\begin{lemma}
		The fourth moment of the normalised one point correlation function $\sigma_{N,b}$ of $H \sim \mathrm{BWE}_{N,b}$ is
		\begin{equation}\label{key}
		m_4(\sigma_{N,b}) =  \frac{6N^2(4b^2+b+1) - 5Nb(4b+1)(b-1)}{3b^2(2N-b+1)^2}
		\end{equation}
		for $b \le \frac{N}{2}$.
	\end{lemma}

	\begin{proof}
		The second moment of $H \sim \mathrm{BWE}_{N,b}$ is simply equal to the number of entries inside the band
		\begin{equation}
		\E \tr (H^2) = \E \tr (H H^\top) = \sum_{i,j} \E h_{ij}^2 = 2Nb-N-b^2+b.
		\end{equation}
		The calculation of the fourth moment is more involved. Taking the dimension $N =  mb$ as an integer multiple of the bandwidth, $H$ can be represented in block form as in equation \eqref{blockform}. Assuming that $b \le \frac{N}{2}$, we can then write
		\begin{equation}
		C = H^2 = \scalemath{0.75}{
			\begin{bmatrix}
			A_1^2 + L_1L_1^\top & A_1L_1+L_1A_2 & L_1L_2 \\
			(A_1L_1+L_1A_2)^\top & L_1^\top L_1 + A_2^2 +L_2L_2^\top & A_2L_2 +L_2A_3 & L_2L_3\\
			(L_1L_2)^\top& (A_2L_2 +L_2A_3)^\top& \ddots &\ddots& \ddots\\
			&(L_2L_3)^\top& \ddots & \ddots & L_{m-2}L_{m-1}\\
			&& \ddots & & A_{m-1}L_{m-1}+L_{m-1}A_{m}\\
			& &  (L_{m-2}L_{m-1})^\top &  (A_{m-1}L_{m-1}+L_{m-1}A_{m})^\top & L_{m-1}^\top L_{m-1} + A_{m}^2\\
			\end{bmatrix}},
		\end{equation}
		or, in block coordinates
		\begin{equation} \begin{split}
		&C_{i,i} \quad = L_{i-1}^\top L_{i-1} + A_{i}^2 +L_{i}L_{i}^\top\\
		&C_{i,i+1} = C_{i+1,i}^\top =  A_iL_i +L_iA_{i+1}\\
		&C_{i,i+2} = C_{i+2,i}^\top = L_{i}L_{i+1}
		\end{split} \end{equation}
		with the convention that $L_0 = L_m = 0$. The fourth moment of $H$ is hence
		\begin{equation} \begin{split}
		\E \tr H^4 = & \E \tr C^2 = \sum_{i,j}^{m}\E \tr (C_{ij}C_{ij}^\top)\\
		=& \sum_{i}^{m} \E\tr( (L_{i-1}^\top L_{i-1} + A_{i}^2 +L_{i}L_{i}^\top)(L_{i-1}^\top L_{i-1} + A_{i}^2 +L_{i}L_{i}^\top)^\top)\\
		+& 2 \sum_{i}^{m-1}\E \tr (( A_iL_i +L_iA_{i+1})( A_iL_i +L_iA_{i+1})^\top)\\
		+& 2 \sum_{i}^{m-2} \E\tr((L_{i}L_{i+1})(L_{i}L_{i+1})^\top).
		\end{split} \end{equation}
		In each of the three sums the terms being summed over are identical, with the exception of the first and last term in the first sum. We therefore have
		\begin{equation} \begin{split}
		\E \tr H^4 =& (m-2) \E\tr((L_{1}^\top L_{1} + A^2  +L_{2}L_{2}^\top)(L_{1}^\top L_{1} + A^2 +L_{2}L_{2}^\top)^\top)\\
		&+ 2 \E \tr ((A^2 + L^\top L)(A^2 + L^\top L)^\top)\\
		&+2(m-1)\E \tr(( A_1L +LA_{2})( A_1L +LA_{2})^\top)\\
		&+2(m-2) \E\tr((L_1L_2)(L_1L_2)^\top).
		\end{split} \end{equation}
		Expanding this out and applying Lemma 1, we find
		\begin{equation}\label{key}
		\begin{split}
		\E \tr H^4 =&\; m \E\tr(A^4) + (6m-8)\E\tr(A^2LL^\top)+2m\E\tr(A^2L^\top L) \\ & + (4m-8)\E\tr(L_1^\top L_1 L_2 L_2^\top)+ (2m-2)\E\tr(LL^\top L L^\top)\\
		=& \frac{27b^2 + 3mb - 18 mb^2 + 24mb^3 - 20b^3-7b}{3}\\
		=& \frac{ 24Nb^2- 18 Nb + 3N - 20b^3 + 27b^2 -7b}{3}.
		\end{split}
		\end{equation}
		The fourth moment of the normalised level density is thus
		\begin{equation}\label{key}
		m_4(\sigma_{N,b}) = \frac{N\E\tr(H^4)}{\left(\E \tr(H^2)\right)^2} = \frac{N(24Nb^2- 18 Nb + 3N - 20b^3 + 27b^2 -7b)}{3(2Nb-N-b^2+b)^2}. 
		\end{equation}
		
	\end{proof}
	\noindent
	We now are ready to prove our main result:
	\begin{theorem}
		In the limit $N\rightarrow \infty$, the critical points of the fourth moment of the normalised one point correlation measure $m_4(\sigma_{N,b})$ of $H \sim \mathrm{BWE}_{N,b}$, considered as a function of the bandwidth, are found at 
		\begin{equation*}\label{key}
		b = \left(\frac{3N}{2}\right)^\frac{1}{2} + o\left(N^{\frac{1}{2}}\right)
		\end{equation*} 
		and
		\begin{equation*}\label{key}
		b = \frac{2}{5}N + o\left(N\right).
		\end{equation*}
	\end{theorem}
	\begin{proof}
		Taking the derivative of $m_4(\sigma_{N,b})$ with respect to $b$, we have
		\begin{equation}\label{key}
		\partial_b m_4(\sigma_{N,b}) = \frac{ -20Nb^4 +(8N^2+34N)b^3+(6N^2-21N)b^2-(12N^3+10N^2-7N)b+6N^3+N^2}{3(2Nb-N-b^2+b)^3}.
		\end{equation}
		As the denominator is $\mathcal{O}(N^3b^3)$, this expression will go to zero as $N \rightarrow \infty$; to see the critical behaviour\footnote{As the numerator is a quartic, it is of course possible to compute the roots exactly, but the resulting expressions are not overly useful as they are nearly a page in length.} we must hence rescale this expression, by say, a factor of $N$, giving us
		\begin{equation}\label{key}
		N\partial_b m_4(\sigma_{N,b}) = \frac{ -20N^2b^4 +(8N^3+34N^2)b^3+(6N^3-21N^2)b^2-(12N^4+10N^3-7N^2)b+6N^4+N^3}{3(2Nb-N-b^2+b)^3}.
		\end{equation}
		Let us first assume that $b \asymp N^\alpha$ with $\frac{1}{3} < \alpha < 1$. Then to leading order we have
		\begin{equation}\label{key}
		N\partial_b m_4(\sigma_{N,b}) = \frac{-12N^4b + 8N^3b^3}{3(2Nb-N-b^2+b)^3} + \mathcal{O}\left(Nb^{-3}+N^{-1}b\right).
		\end{equation}
		This expression will only go to zero if the numerator is $o(N^3b^3)$, which requires
		\begin{equation}\label{key}
		b = \left(\frac{3N}{2}\right)^\frac{1}{2} + o\left(N^{\frac{1}{2}}\right).
		\end{equation}
		Now suppose that $b \asymp N$. Then, to leading order we have
		\begin{equation}\label{key}
		N\partial_b m_4(\sigma_{N,b}) = \frac{8N^3b^3-20N^2b^4}{3(2Nb-N-b^2+b)^3} + \mathcal{O}\left(N^{-1}\right).
		\end{equation}
		Again, this will only go to zero if the numerator is $o(N^3b^3)$, which requires
		\begin{equation}\label{key}
		b = \frac{2}{5}N + o(N).
		\end{equation}
		Finally, if $b \asymp N^\alpha$ with $0 \le \alpha \le \frac{1}{3}$  we have
		\begin{equation}\label{key}
		N\partial_b m_4(\sigma_{N,b}) = \frac{-12N^4b}{3(2Nb-N-b^2+b)^3} + \mathcal{O}\left(Nb^{-3}\right) = \mathcal{O}\left(Nb^{-2}\right),
		\end{equation}
		and so there are no roots in this region.
	\end{proof}	

\section{Evidence of a new phase transition}\label{newphase}

Although existence of critical global behaviour is  surprising, the location of
$b \asymp \sqrt{N}$ is not, given that the conjectured transition in the local/eigenvector statistics occurs in the same region. However, we should not be too hasty to characterise this as ``Poisson/Gaudin--Mehta behaviour'': after all, there is the second critical point to consider and this too must be accounted for. If the same mechanisms are responsible for both the local/eigenvector and global statistics, then there should be some local/eigenvector transition around $b \sim {2\over 5} N$.  We now present numerical evidence that this is indeed the case, at least for the eigenvectors.

	As the fourth moment of the level density is sensitive to positions of all eigenvalues, we should accordingly examine an eigenvector statistic that is similarly dependent on all eigenvectors. An appropriate choice is the total inverse participation ratio
	\begin{equation}\label{key}
	I_4(Q) = \sum_j I_4\left(\psi^{(j)}\right),
	\end{equation} 
	where $I_4\left(\psi^{(j)}\right) = \sum_i |\psi^{(j)}_i|^4$ is the inverse participation ratio of the $j$th eigenvector $\psi^{(j)}$.
	Heuristically, this quantity is on the order of $\frac{N}{l_{\mathrm{av}}}$, where $l_{\mathrm{av}}$ is the average localisation length. For large $N$ we expect $l_{\mathrm{av}} = \mathcal{O}\left(\min\{b^2, N\}\right)$, as contributions from the edge eigenvectors will be negligible. We will thus have $I_4(Q) = \mathcal{O}(1)$ for $b \gg \sqrt{N}$ and $I_4(Q) = \mathcal{O}\left(\frac{N}{b^2}\right)$ for $b \ll \sqrt{N}$ in our first transition phenomena. We also expect that $I_4(Q)$ will become deterministic in the large $N$ limit because the correlation between individual inverse participation ratios $I_4(\psi^{(j)})$ is negligible. Numerical evidence (Figures \ref{fig:ipr1} and \ref{fig:ipr2}) confirms these expectations, and importantly, indicates that a \emph{second transition} occurs around $b \sim \frac{2}{5}N$. This second transition is rather difficult to detect numerically, and is only unequivocally apparent for $N \gtrsim 4000$. 
	\begin{figure}[h!]
		\centering
		\includegraphics[scale=1]{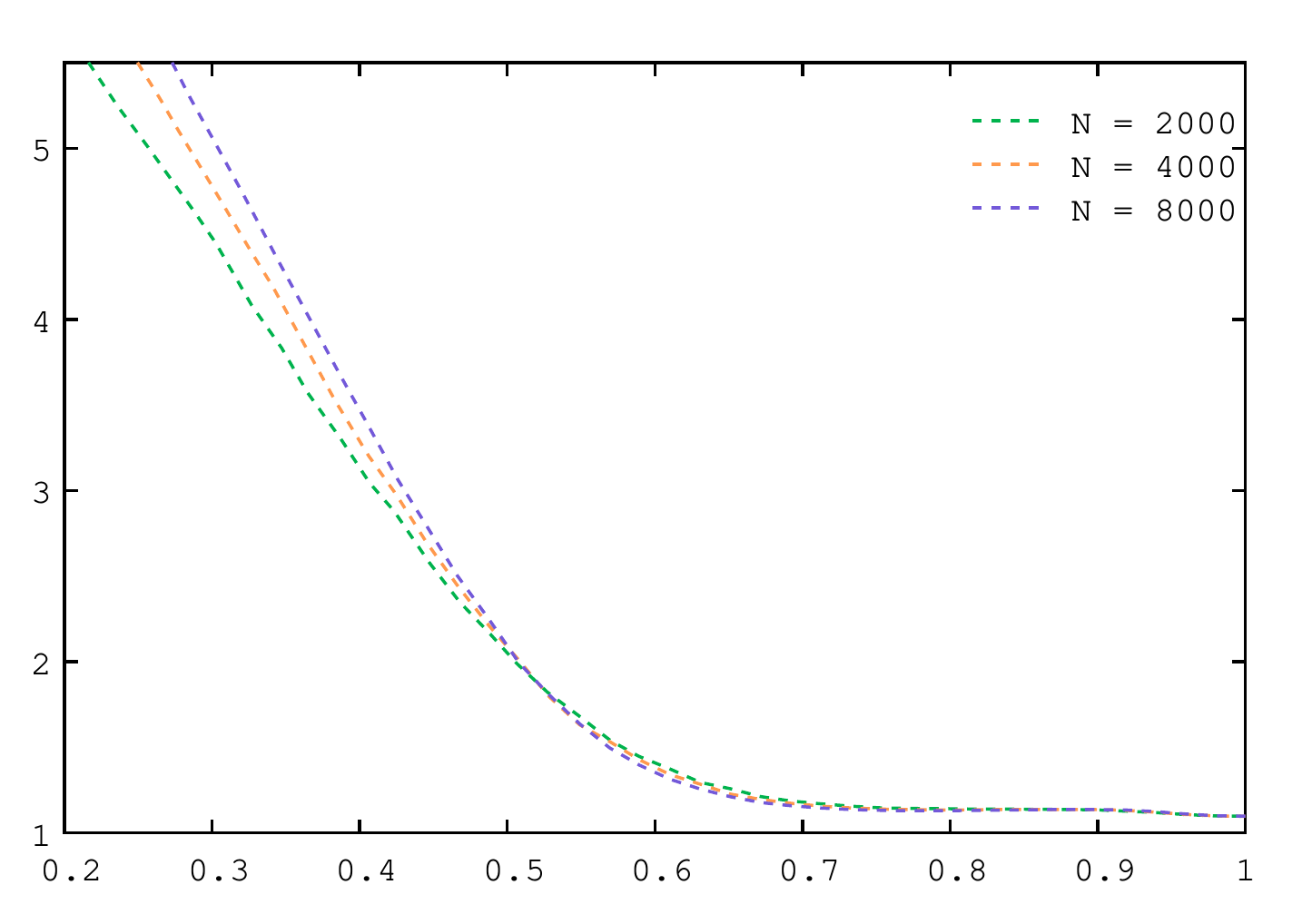}
		\caption{This figure shows the logarithm of the total inverse participation ratio $\log(I_4(Q))$ for 3 sequences of random matrices drawn from a Gaussian $\mathrm{BWE}_{N,b}$ with $N = 2000,4000,8000$, plotted against $\alpha$ with $b = N^{\alpha}$. This figure highlights the first transition: for $\alpha < \frac{1}{2}$, the gradient is approximately $-2\log(N)$, which is consistent with $I_4(Q) = \mathcal{O}\left(\frac{N}{b^2}\right)$; for $\alpha > \frac{1}{2}$, the curves are approximately flat, which is again consistent with $I_4(Q) = \mathcal{O}(1)$ in this region.}
		\label{fig:ipr1}
	\end{figure}
	\begin{figure}
		\includegraphics[scale=1]{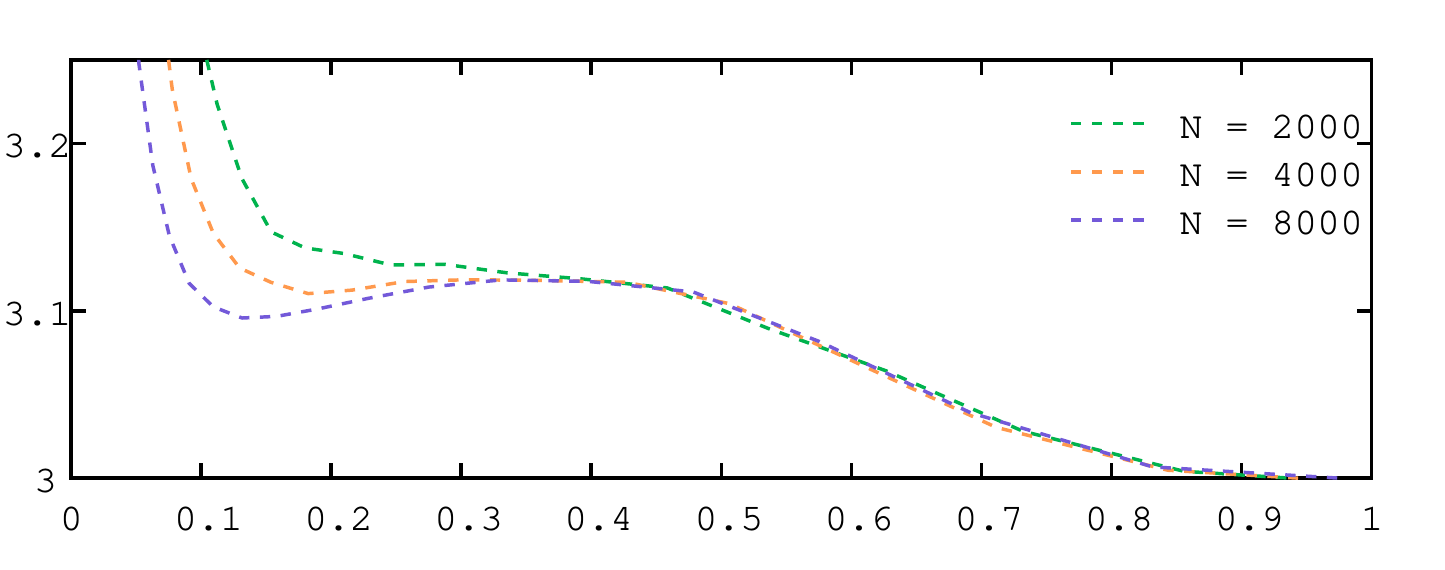}
		\caption{This figure shows the same data as Fig. \ref{fig:ipr1}, but now with a linear scale plotting $I_4(Q)$ against $c$ with $b = cN$. This highlights the second transition: a local maximum in $I_4(Q)$   can be seen to form as $N$ increases around $b \sim \frac{2}{5}N$.}
		\label{fig:ipr2}
	\end{figure}
	
		Why does this second transition occur? The reason, in short, is due to boundary effects.  Until the first transition point $b \asymp \sqrt{N}$, the localisation length grows as $l \asymp b^2$ causing $I_4(Q)$ to decay as $\mathcal{O}\left(\frac{N}{b^2}\right)$. After this point, the Anderson delocalisation mechanism becomes saturated and secondary phenomena affecting $I_4(Q)$ at the scale of $\mathcal{O}(1)$ become visible. How then do boundary effects manifest? Considering the BWE as a Hamiltonian for an Anderson hopping model, we see that the ``width'' of such a conductor is not uniform along its length: nodes near the edges are less connected than those in the middle. Specifically, if $b \le i \le N-b$ then $\mathrm{deg}(u_i) = 2(b-1)$, whereas if $i < b$ or $i > N-b$ then $b-1 \le \mathrm{deg}(u_i)  < 2(b-1)$. In the delocalisation regime, we expect that the eigenvector component fluctuations should be relatively uniform across the homogeneous central section, only changing\footnote{Whether the boundary fluctuations are smaller or larger than in the central region is irrelevant as $I_4(\psi)$ will increase in both cases due to the normalisation $\sum_i |\psi_i|^2 = 1$.} significantly in the boundary regions $i < b$, $i > N-b$. Further, as the delocalisation condition implies that $|\psi^{(j)}_i|^4 = \mathcal{O}{\left( \frac{1}{N^2}\right)}$, the fluctuations at the edge and in the middle are of the same order, i.e., they are different only up to a constant factor. Putting this all together, we see that  boundary effects only become significant when the two regions are of comparable size, i.e., when $b \asymp N$. Continuing this argument, it is easy to see that $I_4(Q)$ will attain local maximum in this regime, because although the boundary effect reduces eigenvector flatness, when $b = N$ we recover a standard Wigner ensemble, whose eigenvectors have uniform fluctuations.
		
		Heuristically, we expect that the nature of the boundary effect 
		on a given eigenvector will depend on its associated eigenvalue: eigenvectors with smaller eigenvalues should tend to have slightly more mass in the boundary region, and those will larger eigenvalues should tend to have slightly more mass in the central region.We can justify this as follows: consider the $2N \times 2N$ matrix
		\begin{equation}\label{key}
		H = \begin{bmatrix} A&0\\
		0&B
		\end{bmatrix}
		\end{equation}
		where $A \sim \mathrm{BWE}_{N,2}$ and $B \sim \mathrm{BWE}_{N,N}$. As $H$ is block diagonal, the eigenvalues of $H$ will be the same as those of $A$ and $B$. From the second moment of the level densities, the eigenvalues of $A$ will be of size $\mathcal{O}(1)$, whereas those of $B$ will be of size $\mathcal{O}(\sqrt{N})$.  Hence, in an interval of size $\mathcal{O}(1)$ around the origin, $H$ will have approximately $\mathcal{O}(N)$ eigenvalues originally from $A$, and $\mathcal{O}(\sqrt{N})$ eigenvalues originally from $B$. As a result, if we choose an eigenvalue at random in this region, with high probability the associated eigenvector will be localised, and located in the region $1 \le i \le N$. Now, consider the matrix $\hat{H} = H + P$, where P is a matrix of zeros except $p_{N,N+1} = p_{N+1,N} = p \sim \mathcal{N}(0,\frac{1}{2})$. This couples the eigenvalues and eigenvectors of $A$ and $B$ together; as a conductor, this represents a model of a thin wire attached to a conducting ``ball''. As we have only introduced a small perturbation, the eigenvectors $\hat{\psi}$ and eigenvalues $\hat{\lambda}$ of $\hat{H}$ should be close to those of $H$. For the eigenvalues, we have from the Lidskii inequality
		\begin{equation}\label{key}
		\sum_{i}^{2N} |\lambda_i - \hat{\lambda}_i|^2 \le \tr (P^2) = 2p^2 = \mathcal{O}(1).
		\end{equation}
		If we again pick an eigenvalue at random from an $\mathcal{O}(1)$ interval around the origin, in the worst case scenario we will have
		\begin{equation}\label{key}
		|\lambda_i - \hat{\lambda}_i|^2 = \mathcal{O}\left(\frac{1}{N}\right),
		\end{equation}
		i.e., $|\lambda_i - \hat{\lambda}_i| = \mathcal{O}\left(\frac{1}{\sqrt{N}}\right)$  with high probability\footnote{Of course, in the worst case scenario for a fixed eigenvalue we would have $|\lambda_i - \hat{\lambda}_i| = 2p^2$, but as we are choosing an eigenvalue at random this would give $|\lambda_i - \hat{\lambda}_i| = 0$ with high probability, which is better; the worst case is when the deviation is spread uniformly across all eigenvalues in the region of interest.}.
		The corresponding eigenvectors are also close to the original. If we assume that the localisation positions $c$ of the original (localised) eigenvectors $\psi$ are uniformly distributed over indices $1,\dots, N$, then 
		\begin{equation}\label{key}
		|\psi_i| \le C e^{\frac{-|i-c|}{l}} 
		\end{equation}
		with $c \le N-\log(N)$ with high probability. As the bandwidth is equal to $2$ in the upper region, we have $l = \mathcal{O}(1)$ and the random $C = \mathcal{O}(1)$. Then
		\begin{equation}
		\begin{split}
		(\hat{H}-\lambda)\psi &= H\psi - \lambda\psi + P\psi\\
		&= P\psi.
		\end{split}
		\end{equation}
		Taking norms, we have
		\begin{equation}
		\begin{split}
		||P\psi||_2^2 &= p^2|\psi_{N}|^2+p^2|\psi_{N+1}|^2\\
		&\le C^2|p|^2\left(e^{\frac{-2|N-c|}{l}}+ e^{\frac{-2|N+1-c|}{l}}\right)\\
		&\le 2C^2|p|^2e^{\frac{-2\left|N-N+\log(N)\right|}{l}}\\
		&= \mathcal{O}\left(N^{\frac{-2}{l}}\right),
		\end{split}
		\end{equation}
		i.e, so $||P\psi||_2 = \mathcal{O}\left(N^{\frac{-1}{l}}\right)$ with high probability.
			\begin{figure}[t!]
				\centering
				\includegraphics[scale=1]{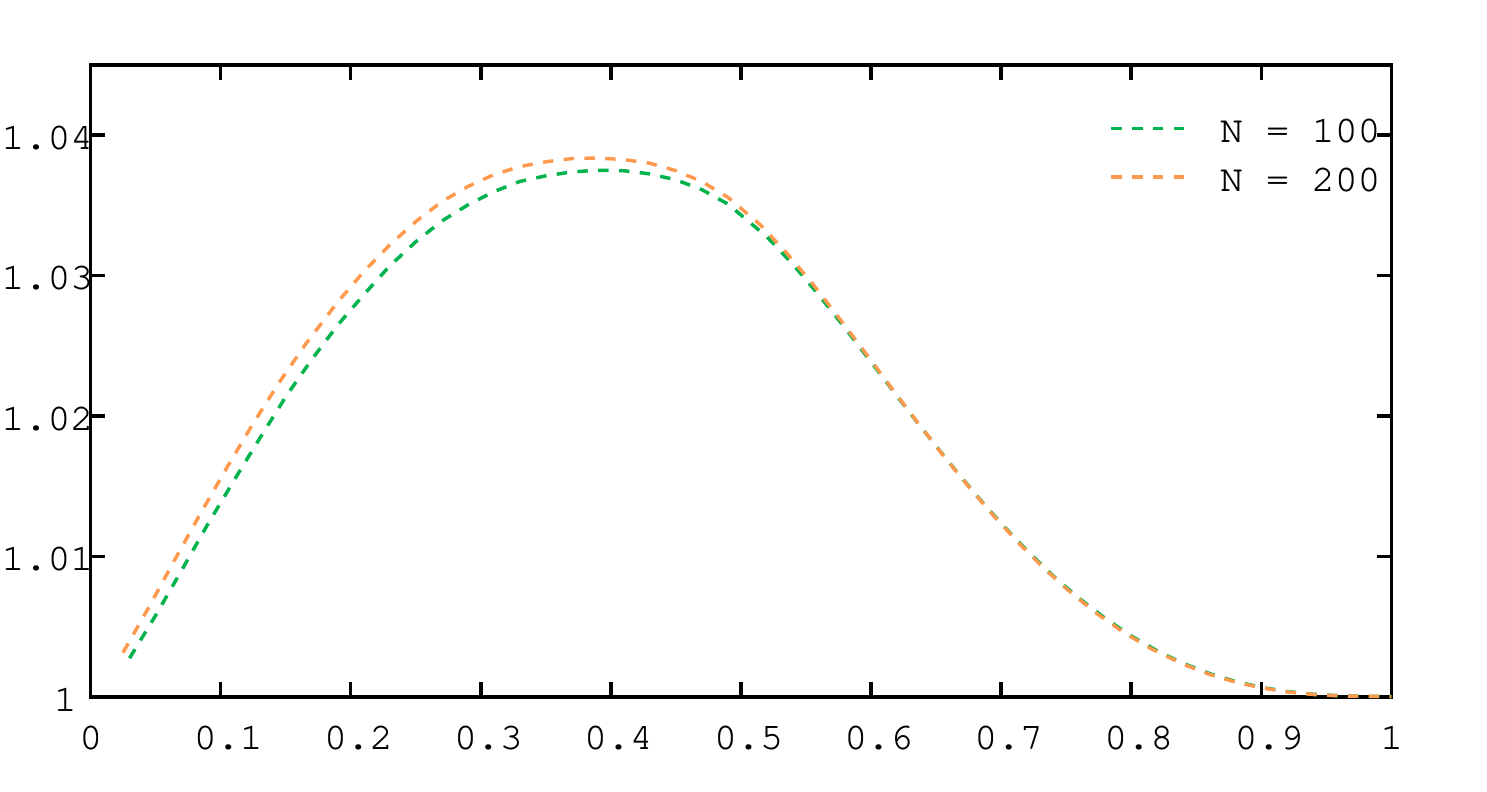}
				\caption{This figure shows $Y(Q)$ against $c$ with $b = cN$ for a Gaussian $\mathrm{BWE}_{N,b}$, $N = 100,200$. A critical point is seen near $b = \frac{2}{5}N$.}
				\label{fig:YQ}
			\end{figure}
		The end result is that small eigenvalues are likely to have eigenvectors supported in regions of the lattice where conductor is thinner (i.e., the bandwidth is smaller). From this, we can also conclude that larger eigenvalues will have eigenvectors supported in regions where the conductor is thicker, as the thin region is already completely occupied. This ``ball and chain'' model thus serves as an extreme example of what we should expect to see in the case of interest, as the conductor is thinner at the edges. Importantly, because of this combined spacial/energy structure, we should be able to separate the boundary effects from the stochastic effects by examining
		\begin{equation}\label{key}
		Y(Q) = \sum_{i,j} \E\left(|\psi_i^{(j)}|^2\right)^2.
		\end{equation}
		Let us first consider the two limiting cases, when $b = 1$ and $b = N$. In both cases, $Y(Q)$ will be equal to 1, because $\E\left(|\psi_i^{(j)}|^2\right)^2 = \frac{1}{N^2}$ for all $i,j$. If our heuristics are correct about the nature of the boundary effect for $1 < b < N$,
		then near the boundary $\E\left(|\psi_i^{(j)}|^2\right)^2 > \frac{1}{N^2}$ for small eigenvalues and $\E\left(|\psi_i^{(j)}|^2\right)^2 < \frac{1}{N^2}$ for large eigenvalues; the net effect should become significant when $b \asymp N$, causing $Y(Q)$ to reach a maximum somewhere in this regime. Figure \ref{fig:YQ} shows a numerical simulation of $Y(Q)$ as a function of the bandwidth for several choices of $N$. The transition is now clear, and remarkably, occurs exactly where we expect it at $b \sim \frac{2}{5}N$. 
	
\section{Conclusion}
\begin{figure}[t!]
	\centering
	\includegraphics[scale=1]{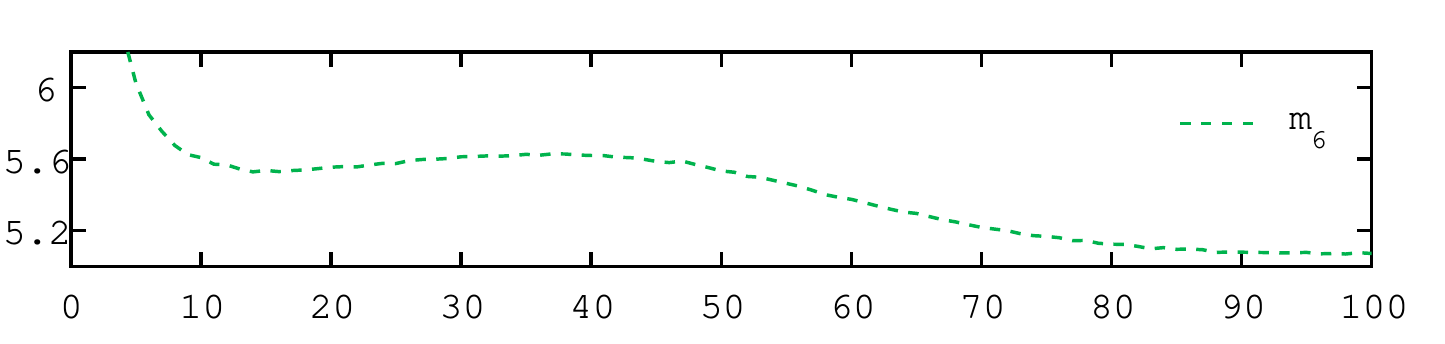}
	\label{fig:m6}
	\includegraphics[scale=1]{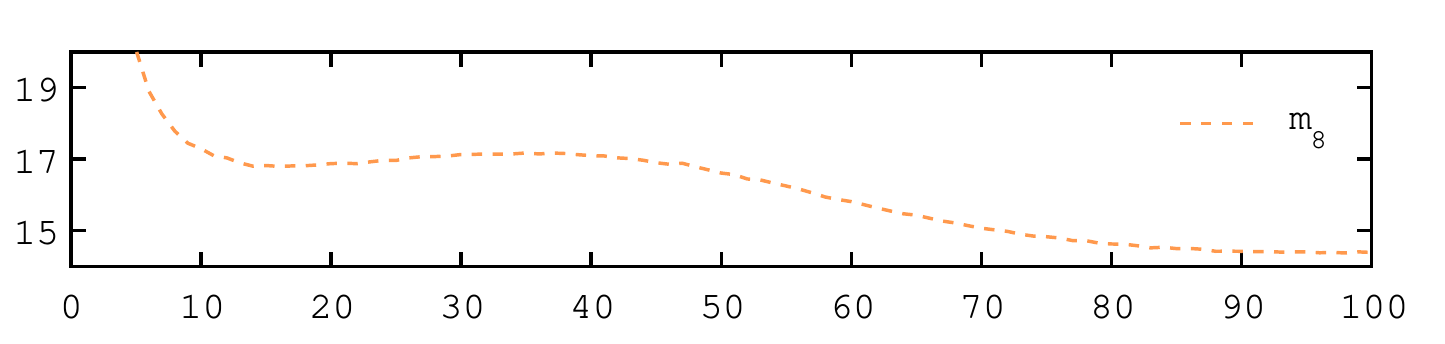}
	\label{fig:m8}
	\includegraphics[scale=1]{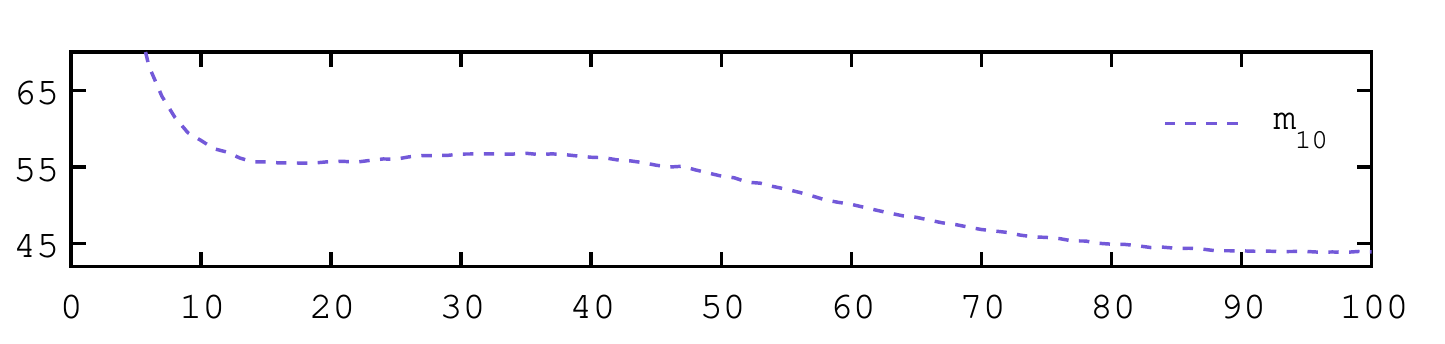}
	\caption{Higher order moments for a Gaussian BWE with $N = 100$. Again, critical points are seen near $b = \sqrt{\frac{3N}{2}}$ and $b = \frac{2}{5}N$. Numerical results for other four moment matching BWEs were similar.}
	\label{fig:mhigh}
\end{figure}
	We have shown that the fourth moment of the normalised eigenvalue level density for banded matrices has a critical point at precisely the  scaling predicted by the Poisson/Gaudin--Mehta conjecture.  A second critical point  led to a conjecture on an as-of-yet unknown second phase transition in the eigenvector localization, whose existence we demonstrated numerically. These results raises several questions:
\begin{enumerate}
	\item Is there evidence of a phase transition in higher order moments?  Numerical results for the Gaussian BWE suggest yes, see Figure \ref{fig:mhigh}.
	\item Is there evidence of a phase transition in the local statistics at $b \sim1 \frac{2}{5}N$? We were unable to find numerical evidence of this, but given the lack of flatness of the eigenvectors in this regime, the existence of such a transition is plausible.
	\item Is there evidence in the global statistics of the phase transition of the largest eigenvalue at $N^{5/6}$ \cite{Sodin2010}?  We have not detected any sign of this transition in numerical experiments, and we suspect that this transition is too localized to have an observable affect on the global statistics.
\end{enumerate}

\section{Appendix}
\begin{proof}[Proof of Lemma 1]
	We first note that the trace and expectation operators commute
	\begin{equation}
	\mathbb{E}\left(\tr\left (A\right )\right) = \tr\left (\mathbb{E}\left(A\right)\right )
	\end{equation}
	for all matrices $A$. This is easy to see as a result of the linearity of expectation:
	\begin{align*}
	\E\left (\tr\left (A\right )\right ) = \E \bigg(\sum_i a_{ii}\bigg) = \sum_i \E\left(a_{ii}\right) = \tr \left (\E\left (A\right )\right).
	\end{align*}
	This result will be useful as it allows us to compute the expectation of independent matrices before taking their product, e.g., 
	\begin{equation}
	\E(\tr(AB)) = \tr(\E(AB)) = \tr(\E(A)\E(B)).
	\end{equation}
	
	The left hand side of equation \ref{A4} may be expanded as
	\begin{equation}
	\E \tr(A^4) = \sum_{i,j}^b \E \tilde{a}_{ij}^{2} = \sum_i^b \E \tilde{a}_{ii}^{2} + \sum_{i\ne j}^b \E \tilde{a}_{ij}^2
	\end{equation}
	where $\tilde{a}_{ij}$ are the elements of the matrix $A^2$. For the diagonal elements we have
	\begin{equation}  \begin{split}
	\E (\tilde{a}_{ii}^2) &= \E \sum_{k = 1}^{b} a_{ik}^4 + \E \sum_{k \ne l}^b a_{ik}^2a_{il}^2\\
	&= b \E a^4 + b(b-1) (\E a^2)^2 \\
	&= 3b + b(b-1).
	\end{split} \end{equation}
	For the off diagonal elements we have
	\begin{equation}
	\E (\tilde{a}_{ij}^2) = \E \sum_{k = 1}^{b} a_{ik}^2a_{jk}^2 + \E \sum_{k \ne l}^b \sum_{l = 1}^{b} a_{ik}a_{jk}a_{il}a_{jl}.
	\end{equation}
	As $i \ne j$ and $k \ne l$, every term of second sum has at least two elements which are different and thus has expectation zero. Therefore
	\begin{equation}
	\E (\tilde{a}_{ij}^2) = b \E a_1^2a_2^2 = b.
	\end{equation}
	The fourth moment of $A$ is hence
	\begin{equation}\label{p1term1} \begin{split}
	\E \tr (A^4) &= \sum_i^b 3b + b(b-1) + \sum_{i \ne j}^b b\\
	&=  b(3b + b(b-1)) + b^2(b-1)\\
	&= 2b^3 + b^2.
	\end{split} \end{equation}
	
	For equation \eqref{A2LLT}, we use the commutativity of the trace and expectation operators together with the independence of $A$ and $L$ to give
	\begin{equation}
	\E \tr (A^2LL^\top) = \tr( \E A^2 \E (L L^\top)).
	\end{equation}
	The expectation of $A^2$ is diagonal as
	\begin{equation} \begin{split}
	\E \tilde{a}_{ii} &= \sum_{k = 1}^b \E a_{ik}^2 = b\\
	\E \tilde{a}_{ij} &= \sum_{k = 1}^b \E a_{ik}\E a_{jk} = 0.
	\end{split} \end{equation}
	Denoting the entries of $L L^\top$ by $\tilde{l}$ and using the fact that $l_{ij}$ is zero if $i \le j$, the diagonal elements are
	\begin{equation}
	\tilde{l}_{ii} = \sum_{k = 1}^b l_{ik}^2 = \sum_{k = 1}^{i-1} l_{ik}^2,
	\end{equation}
	the upper triangular elements $i < j$ are
	\begin{equation}
	\tilde{l}_{ij} = \sum_{k = 1}^{b} l_{ik}l_{jk} = \sum_{k = 1}^{i-1} l_{ik}l_{jk},
	\end{equation}
	and the lower triangular elements $i > j$ are
	\begin{equation}
	\tilde{l}_{ij} = \sum_{k = 1}^{b} l_{ik}l_{ij} = \sum_{k = 1}^{j-1} l_{ik}l_{jk} = \tilde{l}_{ji}.
	\end{equation}
	Taking expectations gives
	\begin{equation}\label{LLtop} \begin{split}
	\E \tilde{l}_{ii} &= i-1\\
	\E \tilde{l}_{ij} &= 0
	\end{split} \end{equation}
	and thus $\E L L^\top$ is also diagonal. We hence have
	\begin{equation}\label{p1term2}\begin{split}
	\E \tr (A^2LL^\top) &= \sum_{i=1}^b \E \tilde{a}_{ii} \E \tilde{l}_{ii} = \sum_{i=1}^b b (i-1) = \frac{b^3 - b^2}{2}.
	\end{split} \end{equation}
	
	Similarly, for equation \eqref{A2LTL} we have
	\begin{equation}
	\E \tr (A^2L^\top L) =  \tr(\E A^2 \E (L L^\top)).
	\end{equation}
	The diagonal entries of $L L^\top$ are
	\begin{equation}
	\tilde{l}_{ii} = \sum_k^b l_{ki}^2 = \sum _{k = i+1}^b l_{ki}^2
	\end{equation}
	and the upper/lower triangular elements are
	\begin{equation}
	\tilde{l}_{ij} = \sum_{k=1}^b l_{ki}l_{kj} = \sum_{k = i+1}^{b}l_{ki}l_{kj} = \tilde{l}_{ji}.
	\end{equation}
	Taking expectations gives
	\begin{equation}\label{LtopL} \begin{split}
	\E \tilde{l}_{ii} &= b-i\\
	\E \tilde{l}_{ij} &= 0
	\end{split} \end{equation}
	so $\E LL^\top$ is diagonal. Hence
	\begin{equation}\label{p1term3}
	\E \tr (A^2L^\top L) = \sum_{i=1}^b \E \tilde{a}_{ii} \E \tilde{l}_{ij} = \sum_{i=1}^b b(b-i) = \frac{b^3-b^2}{2}.
	\end{equation}
	
	Using \eqref{LLtop} and \eqref{LtopL}, equation \eqref{L1L2} is
	\begin{equation}\label{p1term4} \begin{split}
	\E \tr(L_1^\top L_1 L_2 L_2^\top) &= \tr(\E L^\top L \E L L^\top)\\
	&= \sum_{i = 1}^b (i-1)(b-i)\\
	&= \frac{b^3 - 3b^2 + 2b}{6}.
	\end{split} \end{equation}
	
	Finally, equation \eqref{LLLL} may be expressed as
	\begin{equation}
	\E \tr(L L^\top L L^\top) = \E \tr(L L^\top (L L^\top)^\top) =  \sum_{i,j}^b \E \tilde{l}_{ij}^2 = \sum_{i}^b \E \tilde{l}_{ii}^2 + \sum_{i \ne j}^b\E \tilde{l}_{ij}^2
	\end{equation}
	where the $\tilde{l}_{ij}$ are the elements of $L L^\top$. The terms in the first sum are
	\begin{equation} \begin{split}
	\E \tilde{l}_{ii}^2 &= \E \sum_{k=1}^{i-1} l_{ik}^4  + \E \sum_{k\ne m}^{i-1} l_{ik}^2l_{im}^2\\
	&= (i-1)\E l^4 + (i-1)(i-2) (\E l^2)^2\\
	& = 3(i-1) + (i-1)(i-2),
	\end{split} \end{equation}
	and those in the second sum are
	\begin{equation} \begin{split}
	\E \tilde{l}_{ij}^2 & = \E \sum_{k=1}^{i-1} l_{ik}^2l_{ij}^2  + \E \sum_{k\ne m}^{i-1} \sum_{ m}^{i-1} l_{ik}l_{jk}l_{im}l_{jm}\\
	&= (i-1).
	\end{split} \end{equation}
	Thus
	\begin{equation}\label{p1term5} \begin{split}
	\E \tr(L L^\top L L^\top) &=  \sum_{i}^b \E 3(i-1) + (i-1)(i-2) + \sum_{i \ne j}^b\E (i-1)\\
	&= \frac{4b^3 -3b^2 -b}{6}.
	\end{split} \end{equation}
\end{proof}

	\bibliography{newbib}

\end{document}